\documentclass[11pt,a4paper]{article}
\usepackage[margin=1in]{geometry}
\usepackage[T1]{fontenc}
\usepackage[utf8]{inputenc}
\usepackage{tikz}
\usepackage{hyperref}
\usepackage{amsmath,amssymb,amsfonts,amsthm,amscd, bm, mathtools}
\makeatletter
\newtheorem*{rep@theorem}{\rep@title}
\newcommand{\newreptheorem}[2]{%
\newenvironment{rep#1}[1]{%
 \def\rep@title{#2 \ref{##1}}%
 \begin{rep@theorem}}%
 {\end{rep@theorem}}}
\makeatother

\makeatletter
\def\namedlabel#1#2{\begingroup
	#2%
	\def\@currentlabel{#2}%
	\phantomsection\label{#1}\endgroup
}
\makeatother
\newtheorem{thm}{Theorem}[section]
\newreptheorem{theorem}{Theorem}
\newtheorem{prop}[thm]{Proposition}

\newtheorem{lemma}[thm]{Lemma}
\newtheorem{corol}[thm]{Corollary}

\theoremstyle{definition}

\date{August, 2025}

\title{A subbase property for describing edge-end spaces}
\author{Lucas Real\footnote{Institute of Mathematics and Computer Sciences, University of S\~{a}o Paulo, S\~{a}o Paulo, Brazil}}

\begin{document}
\maketitle
\begin{abstract}
	
	In a previous joint work with Aurichi and Magalhães Jr., we showed that the topological spaces arising from the edge-end structure of infinite graphs define a proper subfamily of those obtained through the well-known (vertex-)ends. This result was later recovered by a more general approach due to Pitz, who also stated the problem of finding a purely topological characterization for the class of edge-end spaces. His question reads as an edge-related version of a similar conjecture posed by Diestel in 1992, but there regarding the usual end structure of infinite graphs and which was recently answered also by Pitz via the existence of a suitable clopen subbase. This paper shows how an extra intersection property can be combined with his solution in order to restrict it to the edge-end spaces, hence  stating a topological description for this later family as well.

%In a previous joint work of the author with Aurichi and Magalhães Jr., it was shown that the topological spaces which arise as edge-end spaces of infinite graphs is a proper subfamily of those spaces that arise from the (vertex-)end structure. In a recent work, Pitz introduced a broader approach to this result and also asked for a purely topological description of the former class. His question may be seen as the edge-analogous version of a well-known problem stated by Diestel in 1992 regarding a topological characterization for the family of all (vertex-)end spaces, which Pitz himself also answered in terms of a suitable subbase. In this short note, we revisit this later result and introduce an additional hypothesis that can refine it in order to provide a similar characterization for the class of edge-end spaces. 
\end{abstract}

\section{Introduction}
\paragraph{}
An end structure for a given infinite graph $G$ may be often  understood as a way to formalize the notion of ``directions'' within $G$ or its ``points at the infinity''. Nevertheless, the most usual definitions on that regarding require suitable identifications of the one-way infinite paths in $G$. More precisely, following the textbook \cite[Chapter 8]{diestellivro}, we say that a subgraph $r$ of $G$ defines a \textbf{ray} if its vertex set is a sequence $\{v_n\}_{n\in\mathbb{N}}\subseteq V(G)$ and its edge set is described as $\{v_nv_{n+1}:n\in\mathbb{N}\}$. In this case, we may write $r=v_0v_1v_2\dots$ and say that $v_0$ is the \textbf{starting vertex} of $r$. Besides that, the rays contained in $r$ itself are called now its tails, corresponding then to its infinite connected subgraphs.

Considering the above setting, the more well-known notion of ends is the one proposed by Halin in \cite{halin} and that identifies infinitely vertex-connected one-way infinite paths. Formally, we may say that two rays $r$ and $s$ in $G$ are \textbf{equivalent} if, and only if, $r$ and $s$ have tails in a same connected component of $G\setminus S$ for every finite set of vertices $S\subseteq V(G)$. Under this relation, the equivalence class of $r$ is denoted by $[r]$ and defined to be its \textbf{end}. In many discussions within infinite graph theory, especially regarding locally finite graphs (as surveyed by Diestel in \cite{survey}, for instance), the set $\Omega(G):=\{[r]: r\text{ is a ray of }G\}$ is often understood as a ``boundary of $G$ at the infinity'', whose elements play a similar role as the vertices of $G$ themselves.

However, a different approach to the end structure was later introduced by Hahn, Laviolette and \v{S}irá\v{n} in \cite{edgeends}, who proposed an identification of the infinitely edge-connected one-way infinite paths. Formally, two rays $r$ and $s$ of $G$ are now said to be \textbf{edge-equivalent} if, for every finite set of \textit{edges} $F\subseteq E(G)$, the tails of $r$ and $s$ in $G\setminus F$ lie on the same connected component. In particular, equivalent rays are easily seen to be also edge-equivalent, although the converse implication may not hold (see Figure 1 in \cite{edgeends}, for instance). Due to that, inspired by the previous notation, we now write $[r]_E$ the edge-equivalence class of $r$ and call it its \textbf{edge-end}. As exemplified by the works in \cite{edgeends,edgeconnectivity}, some edge-connectivity properties of $G$ can be highlighted when handling with $\Omega_E(G):=\{[r]_E: r\text{ is a ray of }G\}$.

Nevertheless, irrespective of considering $\Omega(G)$ or $\Omega_E(G)$, the contributions of both quotients to infinite graph theory is often supported by corresponding topological structures. In order to define them, fix a finite set of vertices $S\subseteq V(G)$, a finite set of edges $F\subseteq E(G)$ and a ray $r$ in the given graph $G$. By definition of the previous equivalence relations, the connected components of $G\setminus S$ and $G\setminus F$ in which $r$ has a tail are well-defined regardless of the end and the edge-end of $r$, respectively. Then, we may write $C(S,[r])$ for such connected component of $G\setminus S$, while $C_E(F,[r]_E)$ shall refer to the one in $G\setminus F$ containing the tail of $r$. Considering that, $\Omega(S,[r]):=\{[s]\in \Omega(G): C(S,[s]) = C(S,[r]) \}$ sets a basic open set around $[r]$ in a topology for $\Omega(G)$, which is then called the \textbf{end space} of $G$. Similarly, $\Omega_E(F,[r]_E):=\{[s]_E\in \Omega_E(G): C_E(F,[s]_E) = C_E(F,[r]_E)\}$ defines an open basic neighborhood around $[r]_E$ in a topology for $\Omega_E(G)$, now said to be the \textbf{edge-end space} of $G$.

Despite the combinatorial motivations for the study of $\Omega(G)$ and $\Omega_E(G)$, the topological approach to these structures also offer interesting problems from a set-theoretical viewpoint. As a well-known discussion on that regard, Diestel asked in \cite{diestel1992} for a characterization of the family $\Omega:=\{\Omega(G): G\text{ graph}\}$ into purely topological terms, which was achieved by Pitz in \cite{caracterizacao} only more than 30 years later. On the other hand, Aurichi, Magalhães Jr. and Real showed in \cite{compaulo} that every topological space arising from the edge-end structure of an infinite graph could also be obtained through the end space of a possible another one, but that the converse statement does not hold. Through a more general approach, this inclusion ``$\Omega_E\subsetneq \Omega$'' was also recently re-obtained by Pitz in his paper \cite{TreeCutPartition}, whose Problem 1 then asks for purely topological description of the (strictly) smaller class $\Omega_E:=\{\Omega_E(G): G\text{ graph}\}$. As our main result, which is strongly supported by the own characterization for the family $\Omega$ due to Pitz in \cite{caracterizacao}, we address this question through the subbase criteria below:       
  
  \begin{thm}\label{main}
  	A topological space $X$ is homeomorphic to the edge-end space of some graph if, and only if, it admits a clopen subbase $\mathcal{C}$ which is nested, noetherian, hereditarily complete and verifies the following \textbf{singleton intersection property}: \begin{center}
  		For every strictly $\subseteq-$decreasing chain $C_0\supsetneq C_1\supsetneq C_2\supsetneq C_3 \supsetneq \dots$ in $\mathcal{C}$, the intersection $\bigcap_{n\in\mathbb{N}}C_n$ contains an unique element.
  	\end{center} 
  \end{thm}

The proof of Theorem \ref{main} itself is the core of Section \ref{FinalProof}, where it is split into two parts. First, Proposition \ref{MinhaRepresentacao} writes $\Omega_E$ in terms of ray spaces of suitable order trees, reading thus as an edge-related counterpart to Theorem 1 in \cite{representacao}. This intermediate step relies on a recent representation result for edge-end spaces due to Pitz in \cite{TreeCutPartition}, whose combinatorial requirements are reviewed throughout Section \ref{sec:GraphTrees}. On the other hand, the appropriate spaces claimed by Proposition \ref{MinhaRepresentacao} are constructed with tools from Section \ref{sec:order}, most of which also follows the work of Pitz in \cite{caracterizacao}. Finally, Proposition \ref{CaracterizacaoDeFato} concludes the proof of our main result after combining Proposition \ref{MinhaRepresentacao} with a central characterization theorem in \cite{caracterizacao}, from where the above subbase axioms other than the singleton intersection property were inherited.

\section{Background from the previous literature}\label{sec:background}
\paragraph{}

As suggested in the introduction, Theorem \ref{main} strongly relies on the recent progresses within the literature regarding end and edge-end spaces. Aiming to compile most of such background, this section summarizes key definitions and discussions from the papers \cite{representacao,caracterizacao,TreeCutPartition}. In particular, we start by revisiting the work of Kurkofka and Pitz in \cite{representacao} concerning a representation result for the class of end spaces. There, they concluded that the family $\Omega = \{\Omega(G): G\text{ graph}\}$ may be obtained also from an order-theoretical perspective. In order to precise this remark, we first present the following set of terminologies regarding a partially ordered set $T=(T,\leq)$:

\begin{itemize}
	\item  We call $A\subseteq T$ an \textbf{antichain} of $T$ if it comprises pairwise incomparable elements. This is a counterpart to \textbf{totally ordered} subsets of $T$, whose elements are pairwise comparable regarding $\leq$. In its turn, we say that $C\subseteq T$ is \textbf{cofinal} in some $C'\subseteq T$ if for each $t'\in C'$ there is $t\in C$ such that $t'\leq t$;
	\item For each $t\in T$ we define the sets $\lceil t \rceil_T :=\{s \in T : s \leq t\}$, $\mathring{\lceil t\rceil}_T :=\lceil t\rceil \setminus \{t\} $ and $\lfloor t\rfloor_T :=\{s\in T : s \geq t\}$. When no ambiguity arises and the partial order is clear from the context, we may omit the subscript ``$T$'' in this previous notation. Then, $D\subseteq T$ is said to be \textbf{down-closed} in $T$ if $\lceil t\rceil \subseteq D$ for every $t\in D$;
	\item $T$ is called an (order) \textbf{tree} if it contains a $\leq-$minimum element $t_0\in T$ and $\lceil t\rceil$ is well-ordered by $\leq$ for each $t\in T$. The node $t_0$ is defined as the \textbf{root} of $T$ in this case, while the order-type $\mathrm{ht}(T)$ of $\mathring{\lceil t\rceil}$ for some $t\in T$ is said to be its \textbf{height}. If $\alpha$ is a given ordinal, we then set the $\alpha-$\textbf{level} of $T$ as $\mathcal{L}_{\alpha}(T):=\{t\in T : \mathrm{ht}(t) = \alpha\}$, so that $\min \{\alpha : \mathcal{L}_{\alpha}(T)=\emptyset\}$ is now called the \textbf{height} of the own tree;
	\item For each node $t$ of a tree $T$, the $\leq-$minimum elements from $\lfloor t\rfloor \setminus \{t\}$ are called its \textbf{successors}, which particularly define an antichain in $T$. When there are at least two distinct of these successors, we even say that $t$ is a \textbf{branching node} of $T$;
	\item A \textbf{path} in $T$ corresponds to any down-closed and totally ordered subset $P\subseteq T$, then also being well-ordered by $\leq$. If $P$ has limit order-type (or, equivalently, contains no $\leq-$maximum element), it is also called a \textbf{ray} of $T$.  The families of paths and rays in $T$ shall be denoted by $\mathcal{P}(T)$ and $\mathcal{R}(T)$ respectively;
	\item Finally, we say that a ray $R\in\mathcal{R}(T)$ in a tree $T$ is \textbf{branching} if there is a node $t\in T$ such that $R = \mathring{\lceil t\rceil}$. In this case, $t$ is called a \textbf{top} of $R$ and, by definition, it must have limit height in $T$.    
\end{itemize}

When identifying subsets of a given order tree $T$ with their identity functions, we may see $\mathcal{P}(T)$ and $\mathcal{R}(T)$ as subspaces of the product topology $2^T$. Considering this structure, $\mathcal{P}(T)$ and $\mathcal{R}(T)$ are referred in \cite{caracterizacao} as the \textbf{path} and \textbf{ray spaces} of $T$ respectively. In particular, a subbase for $\mathcal{P}(T)$ may be given by $\{[t]_T: t\in T\}\cup \{\mathcal{P}(T)\setminus [t]_T: t\in T\}$, where we set $[t]_T:=\{P\in \mathcal{P}(T): t\in P\}$ for every $t\in T$. In its turn, according to Lemma 2.1 in \cite{caracterizacao}, a basic open set in $\mathcal{R}(T)$ around a ray $R$ may be written as $[t,F]_T:=\{P\in \mathcal{R}(T): t\in P\text{ but }s\notin P\text{ for every }s\in F\}$, where $t$ is a fixed node of $R$ and $F$ comprises finitely many tops of this ray. As before, the subscript ``$T$'' shall be omitted from these notations when the corresponding order tree is uniquely determined by the context. Within this setting, one of the main results due to Kurkofka and Pitz in \cite{representacao} can be stated as follows:

\begin{thm}[\cite{representacao}, Theorem 1]\label{RepresentacaoPitz}
	The topological spaces from $\Omega:=\{\Omega(G) : G\text{ graph}\}$ are homeomorphic to precisely those that arise as ray spaces of special order trees, namely, trees that can be written as a union of countably many of its antichains. 
\end{thm}

Later in Section \ref{FinalProof}, Proposition \ref{MinhaRepresentacao} may be read as a restriction of the above result to the class of edge-end spaces, providing there a correspondence with a particular family of ray spaces of (special) order trees. Before that, however, we shall first revisit the work of Pitz in \cite{caracterizacao} towards a purely topological description for the family of spaces characterized by Theorem \ref{RepresentacaoPitz}. In particular, once fixed a family $\mathcal{C}$ comprising non-empty subsets of a topological space $X$, we recall the following definitions from subsections 2.5 and 3.1 in \cite{caracterizacao}:  

\begin{itemize}
	\item $\mathcal{C}$ is said to be a \textbf{clopen subbase} if $\mathcal{C}\cup \{X\setminus C : C \in \mathcal{C}(X)\}$ is actually a subbase of $X$;
	\item We call $\mathcal{C}$ \textbf{nested} if $C\subseteq D$ or $D\subseteq C$ whenever $C,D \in \mathcal{C}$ intersect. If the elements of $\mathcal{C}$ are pairwise disjoint, we simply call it a \textbf{disjoint} family.;
	\item We say that $\mathcal{C}$ is $\sigma-$\textbf{disjoint} if we can write $\mathcal{C}= \bigcup_{n\in\mathbb{N}}\mathcal{A}_n$ for some countable collection $\{\mathcal{A}_n\}_{n\in\mathbb{N}}$ of disjoint subfamilies of $\mathcal{C}$;
	\item $\mathcal{C}$ is called \textbf{noetherian} if there is no $\subseteq-$strictly ascending chain in $\mathcal{C}$. In other words, for each sequence in $\mathcal{C}$ of the form $C_0\subseteq C_1\subseteq C_2 \subseteq \dots$ there is some $n\in\mathbb{N}$ such that $C_n=C_m$ for every $m\geq n$;
	\item We say that $\mathcal{C}$ is \textbf{centered} if $\bigcap \mathcal{F}\neq \emptyset$ whenever $\mathcal{F}\subseteq \mathcal{C}$ is a non-empty and finite subset. If any centered subfamily of $\mathcal{C}$ has itself non-empty intersection, then $\mathcal{C}$ is further said to be \textbf{complete};
	\item We say that $\mathcal{C}$ is \textbf{hereditarily complete} if $\{C\cap Y : C \in \mathcal{C}\}$ is complete in $Y$ for each closed subspace $Y\subseteq X$;
	\item An order-theoretical tree $T = (T,\leq)$ is said to be a $\mathcal{C}-$\textbf{tree} if there is an associated function $f: T\to \mathcal{C}$ such that, given two nodes $s,t\in T$, we have $f(t)\subseteq f(s)$ if $s\leq t$ and $f(t)\cap f(s) = \emptyset$ if $t$ and $s$ are incomparable regarding $\leq$.
\end{itemize}

In particular, the subbase properties mentioned by Theorem \ref{main} are now completely set throughout the above items. In fact, most of its hypothesis are inherited from the characterization results of Pitz in \cite{caracterizacao}, whose central application reads as follows: a topological space $X$ is homeomorphic to the end space of some graph if, and only if, it admits a clopen subbase that is nested, noetherian, hereditarily complete and $\sigma-$disjoint. Supported by Theorem \ref{RepresentacaoPitz}, a detailed version of this correspondence could be recalled from Theorem 3.6 in \cite{TreeCutPartition} as stated below:

\begin{thm}[\cite{caracterizacao}, Theorem 3.6]\label{caracterizacaoPitz}
	A topological space $X$ is homeomorphic to the ray space of some order tree $T$ if, and only if, it admits a clopen subbase $\mathcal{C}$ which is noetherian, nested and hereditarily complete. In this case, $T$ can be chosen to be a $\mathcal{C}-$tree such that, denoting by $f: T\to \mathcal{C}$ the associated map, an homeomorphism $e: X\to \mathcal{R}(T)$ may be defined as $$x\mapsto e(x):=\{t\in T: x\in f(t)\}$$ for every $x\in X$. Moreover, $T$ is special if, and only if, $\mathcal{C}$ is $\sigma-$disjoint.  
\end{thm}

Again, we shall later explain on Section \ref{FinalProof} how Theorems \ref{RepresentacaoPitz} and \ref{caracterizacaoPitz} can be restricted to the theory of edge-end spaces. However, this still requires a final recent result from the previous literature on the subject. In order to introduced it, call $C$ a \textbf{region} of a fixed graph $G$ if $C$ is the connected component of $G\setminus F$ for some finite set of edges $F\subseteq E(G)$. It is clear that $V(C)\subseteq V(G)$ in this case, but is also not difficult to regard $\Omega_E(C)$ as a subset of $\Omega_E(G)$: in fact, since $F$ is finite, two rays in $C$ are edge-equivalent in $C$ if, and only if, so they are in $G$. In other words, each equivalence class $[r]_E\in \Omega_E(C)$ gives rise to the unique edge-end of $r$ in $G$, allowing us to write $\Omega_E(C)\subseteq \Omega_E(G)$. When denoting $\|G\|:=V(G)\cup \Omega_E(G)$ and $\|C\|:=V(C)\cup \Omega_E(G)$, this just formalizes the inclusion $\|C\|\subseteq \|G\|$. Hence, following the definition of Pitz in \cite{TreeCutPartition}, we may set $\{\|C\|: C\text{ is  a region of }G\}$ as a basis for a topology over $\|G\|$. In particular, the edge-end space $\Omega_E(G)$ is realized as the corresponding subspace of $\|G\|$. Within those terms, the characterization below due to Pitz in \cite{TreeCutPartition} finishes the required theoretical background for approaching our Theorem \ref{main}:

\begin{thm}[\cite{TreeCutPartition}, Theorem 2]\label{RepresentacaoEdgeEndsPitz}
	A topological space $X$ is homeomorphic to the edge-end space of some graph if, and only if, there is a graph-theoretical tree $T$ such that $X$ is homeomorphic to some subspace $K\subseteq \|T\|$ containing $\Omega_E(T)$.
\end{thm}

\section{Path spaces of graph-theoretical trees} \label{sec:GraphTrees}
\paragraph{} 
In this section, we aim to bring an order-theoretical interpretation to the topology of $\|T\|$ when $T$ is a given graph-theoretical tree, since this is a first step in order to connect Theorem \ref{RepresentacaoEdgeEndsPitz} with Theorems \ref{RepresentacaoPitz} and \ref{caracterizacaoPitz}. In fact, as an acyclic connected graph, recall that $T$ admits a natural tree-order $\leq$ after a root $v_0\in V(T)$ is fixed. More precisely, for vertices $u,v \in V(T)$ we set $u \leq v$ if, and only if, $u$ belongs to the unique path $P_v$ in $T$ whose endpoints are $v$ and $v_0$. In this case, $P_u\subseteq P_v$ and $P_v\setminus P_u$ is still a path in $T$. As also highlighted in \cite{order}, $\leq$ indeed gives to $(V(T),\leq)$ a structure of an order tree, in which $\lceil v\rceil = P_v$ for every $v\in V(T)$. Hence, the finiteness of each $P_v$ particularly implies that the height of $T$ is $\omega$. Summarizing other connectivity properties of $T$ in terms of its ordering $\leq$, the result below compiles some routine exercises within graph theory, but which we detail here for convenience of the reader:

\begin{lemma}\label{BasicPropertiesTree}
	Let $T$ be a graph-theoretical tree rooted at some $v_0\in V(T)$. If $\leq$ denotes the corresponding tree order, then:
	\begin{itemize}
		\item[$a)$] $T\setminus e$ has two connected components for each $e\in E(T)$. If we write $e = uv$ for vertices $u,v\in V(T)$ such that $u<v$, then one of these components is induced by $\lfloor v\rfloor$;
		\item[$b)$] Each edge-end of $T$ contains an unique representative starting at $v_0$.
		\item[$c)$] The family $\mathcal{C}(T):=\{\|C\|: C \text{ is a connected component of }T\setminus e\text{ for some }e\in E(T)\}$ is a subbase for the topology of $\|T\|$.
	\end{itemize}
	Moreover, an order tree $T:=(T,\leq)$ is order-isomorphic to a rooted graph-theoretical tree as above if, and only if, $T$ has height bounded by $\omega$.
\end{lemma}
\begin{proof}
	
	Any graph-theoretical tree $T$ is minimally connected, in the sense that $T$ is connected but $T\setminus e$ is disconnected for every edge $e\in E(T)$. In this case, each connected component of $T\setminus e$ contains one of the endpoints of $e$. When writing $e=uv$ for vertices $u,v\in V(T)$ such that $u<v$, the connectedness of $\lfloor v \rfloor$ as in $a)$ follows from Lemma 1.5.5 in \cite{diestellivro} for instance.

	Aiming to conclude $b)$, note first that each edge-end of $T$ admits a representative starting at $v_0$ because $T$ is connected. However, if $r$ and $r'$ are two distinct rays starting at $v_0$, Lemma 1.5.5 in \cite{diestellivro} also ensures the existence of $\leq-$incomparable vertices $u,u'\in V(T)$ such that $\lfloor u\rfloor$ and $\lfloor u'\rfloor$ contain tails of $r$ and $r'$ respectively. In fact, $u$ and $u'$ can be chosen as successors of the $\leq-$maximum element from the intersection $r\cap r'$, which contains $v_0$ but is finite since $r,r'$ are distinct rays of an acyclic graph. Then, by the previous item, $r$ and $r'$ have their tails on different connected components of $T\setminus e$, where $e\in E(T)$ may be fixed as the unique edge in $T$ of the form $vu$ with $v<u$. Hence, $r$ and $r'$ do not belong to the same edge-end from $\Omega_E(T)$.

	In order to prove $c)$, first note that each $\|C\|\in \mathcal{C}(T)$ is clearly an open set in $\|T\|$ by definition of their topologies. Conversely, let $C$ be a region of $T$ and fix a corresponding finite set of edges $F\subseteq E(T)$ such that $C$ is a connected component of $T\setminus F$. We shall now verify the equality $\|C\| = \bigcup_{e\in F}\|C_e\|$, where $C_e$ denotes the connected component of $T\setminus e$ containing $C$ for each edge $e\in F$. In particular, any edge-end of $T$ with representatives in $C$ also has representatives in $C_e$, because $V(C)\subseteq V(C_e)$ precisely by the choice of $C_e$. Hence, the inclusion $\|C\|\subseteq \bigcup_{e\in F}\|C_e\|$ is immediate. On the other hand, fix a pair of vertices $u,v\in \bigcap_{e\in F}V(C_e)$ and denote by $P$ the unique path in $T$ containing them as endpoints. This uniqueness implies that $P\subseteq C_e$ for each $e\in F$, once $C_e$ is connected and also contains both $u$ and $v$. Since $e\notin E(P)$ for every $e\in F$, we showed that $\bigcap_{e\in F}V(C_e)$ describes the connected component of $T\setminus F$ containing $C$ and, therefore, that $C = \bigcap_{e\in F}V(C_e)$. If $r$ is ray of $T$ which has a tail in $C_e$ for each $e\in F$, the finiteness of $F$ also ensures that we can fix such a tail $r'$ for which $E(r')\cap F=\emptyset $. Hence, $V(r')\subseteq \bigcap_{e\in F}V(C_e)\subseteq C$ because $r$ (and thus also $r'$) has infinite intersection with $C_e$ for every $e\in F$. In other words, concluding the inclusion $\bigcap_{e\in F}\Omega_E(C_e)\subseteq \Omega_E(C)$, we showed that $[r]_E\in \Omega_E(C)$. Hence, $\|C\| = \bigcap_{e\in F}\|C_e\|$.

	Finally, in order to prove the ``Moreover-''part of the statement, we already highlighted that $(T,\leq)$ has height $\omega$ if the ordering $\leq$ is obtained after fixing a root for a graph-theoretical tree $T$. Conversely, if $T:=(T,\leq)$ is an order tree rooted at some $t_0\in T$ and whose nodes have finite height, then each $t\in T\setminus \{t_0\}$ has a well defined \textit{predecessor} $t':=\max \mathring{\lceil t\rceil} $. Then, consider the graph on $T$ whose set of edges is given by $E(T):=\{tt': t\in T, t'\text{ is the predecessor of }t\}$. If $t_0\leq t_1 \leq t_2\leq \dots \leq t_n$ expresses the down-closure of some node $t=t_n\in T$, it clearly follows that $t_{i}$ is a predecessor of $t_{i+1}$ for each $0\leq i < n$, so that $\{t_i\}_{i=0}^n$ induces a graph-theoretical path in $T$ connecting $t$ to $t_0$. In other words, the edge set $E(T)$ turns $T$ into a connected graph. Moreover, for each $e=tt' \in E(T)$ such that $t< t'$, it also follows from the definition of $E(T)$ that $e$ is the unique edge which has one endpoint in $\lfloor t'\rfloor = \{s\in T : s \geq t'\} $ and the other in $T\setminus \lfloor t'\rfloor$. Hence, $T\setminus e$ gives rise to a disconnected graph for every $e\in E(T)$, proving that $T$ is minimally connected and, thus, a graph-theoretical tree. In this case, $\leq$ corresponds to the tree-order obtained after fixing $t_0$ as a root because, for each node $t\in T$, the unique graph-theoretical path in $T$ having $t$ and $t_0$ as endpoints was already verified to be induced by $\lceil t\rceil$.     
\end{proof}

In its turn, item $b)$ from the above result motivates an identification between edge-ends of a graph-theoretical tree $T$ and the rays in $\mathcal{R}(T)$. More precisely, again denoting by $\leq$ the tree-order of $T$ corresponding to a fixed root $v_0\in V(T)$, an order-theoretical path $P\in \mathcal{P}(T)$ can be completely described through one of the following items:

\begin{itemize}
	\item If $P$ is finite, then its $\leq-$maximum vertex $v$ verifies that $P = \lceil v \rceil$, since $P$ itself is down-closed regarding $\leq$;
	\item If $P$ is infinite (and thus an element of $\mathcal{R}(T)$), then $P$ has order type $\omega$ because this ordinal is an upper bound for the own height of $T$. When fixing an enumeration $P=\{v_n\}_{n\in\mathbb{N}}$ such that $v_n<v_{n+1}$ for every $n\in\mathbb{N}$, it clearly follows that $v_{n+1}$ is the unique element of $\lceil v_{n+1}\rceil \setminus \lceil v_n\rceil$. Once both sets $\lceil v_{n+1}\rceil$ and $\lceil v_n\rceil$ describe graph-theoretical paths in $T$ by definition of $\leq$, it turns out that $P= v_0v_1v_2\dots$ is a presentation of a ray in $T$ starting at the root. 
\end{itemize}

As suggested by the above items, hence, we finish this section with a remark that the spaces $\mathcal{P}(T)$ and $\|T\|$ are homeomorphic. In particular, this formalizes a topological correspondence between order-theoretical and graph-theoretical rays in $T$:

\begin{lemma}\label{PathTopologies}
	Fix a graph-theoretical tree $T$ rooted at some $v_0\in V(T)$ and denote by $\leq$ the corresponding tree-order. Let $\sigma : \mathcal{P}(T)\to \|T\|$ be the map that assigns to each order-theoretical finite path $P\in\mathcal{P}(T)\setminus \mathcal{R}(T)$ the element $\sigma(P):=\max P$ and to each order-theoretical ray $P\in \mathcal{R}(T)$ the edge-end $\sigma(P):=[P]_E$. Then, $\sigma$ is an homeomorphism verifying $\sigma(\mathcal{R}(T)) = \Omega_E(T)$. 
\end{lemma}
\begin{proof}
	Since $\sigma(\lceil v \rceil) = \max \lceil v \rceil = v$ for every vertex $v\in V(T)$, it is clear that $\sigma$ establishes a bijection between the finite order-theoretical paths from $\mathcal{P}(T)$ and the vertices of $T$. Similarly, the second bullet point of Lemma \ref{BasicPropertiesTree} claims that $\sigma(P) = [P]_E\neq [Q]_E=\sigma(Q)$ if $P,Q\in \mathcal{R}(T)$ are different order-theoretical rays, which then can be seen as distinct graph-theoretical rays starting at the root of $T$. Conversely, if $r= v_0v_1v_2\dots$ is a graph-theoretical ray in $T$ starting at $v_0$, then the unique path in $T$ connecting each $v_n$ to $v_0$ is the one already contained in $r$ and given by $P_n:=v_0v_1\dots v_n$. Observing that $\lceil v_n\rceil = P_{n}\subseteq P_{{n+1}}=\lceil v_{n+1}\rceil$ for each $n\in\mathbb{N}$ by definition of $\leq$, the equalities $r = \bigcup_{n\in\mathbb{N}}\lceil v_n\rceil = \{v_0<v_1<v_2<\dots\}$ now express $r$ as a totally ordered and down-closed subset of $T$, which clearly has $\omega$ as limit order type. In other words, $r\in \mathcal{R}(T)$ defines a ray of $T$ also in a order-theoretic context. Hence, $\sigma$ also restricts to a bijection between $\mathcal{R}(T)$ and $\Omega_E(T)$.

	It thus remains to show that $\sigma$ is a homeomorphism. To that aim, fix a vertex $v\in V(T)$ and recall the notation $[v]:=\{P\in \mathcal{P}(T): v\in P\}$. Since each element from $\mathcal{P}(T)$ is down-closed, note that a finite order-theoretical path from $\mathcal{P}(T)$ contains $v$ if, and only if, $\sigma(P) = \max P$ belongs to $\lfloor v\rfloor$. Similarly, an order-theoretical ray $P\in \mathcal{R}(T)$ contains $v$ if, and only if,  $P\setminus \lceil v\rceil \subseteq \lfloor v \rfloor $. Observing that $C_v:=\lfloor v\rfloor$ and $T\setminus C_v$ define two regions of $T$ by Lemma \ref{BasicPropertiesTree}, these equivalences formalize the equalities $\sigma([t])=\|C_v\|$ and $\sigma(\mathcal{P}(T)\setminus [t]) = \|T\setminus C_v\|$ for every $v\in V(T)$. Conversely, given an edge $e\in E(T)$ written as $e = uv$ for some pair $u,v\in V(T)$ such that $u<v$, the first bullet point from Lemma \ref{BasicPropertiesTree} again ensures that $C_v=\lfloor v\rfloor$ and $T\setminus C_v$ define the two connected components of $T\setminus e$. Hence, once $\sigma([v_0]) = \sigma(\mathcal{P}(T)) = \|T\|$, we just proved that $\sigma$ establishes a bijection between the subbase of $\mathcal{P}(T)$ given by $\{[v],\mathcal{P}(T)\setminus [v]: v\in V(T)\}$ and the subbase $\mathcal{C}(T)\cup\{\|T\|,\emptyset\}$ for $\|T\|$ as item $c)$ in Lemma \ref{BasicPropertiesTree} describes. Therefore, $\sigma$ is an homeomorphism. 
\end{proof}

\section{Suitable path and ray spaces}\label{sec:order}
\paragraph{}

This section revisits a key construction due to Pitz in \cite{caracterizacao} when proving his Theorem 3.9, which sets the class of all path spaces as the one that also arises from the compact ray spaces of order trees. Approaching this result is useful since Theorem \ref{RepresentacaoEdgeEndsPitz} is stated in terms of a path space topology (by Lemma \ref{PathTopologies}), while Theorems \ref{RepresentacaoPitz} and \ref{caracterizacaoPitz} rather concern characterizations of suitable ray spaces. In order to connect these structures, we first start with the standard remark below:

\begin{lemma}\label{InclusionwiseMinimalRay}
	Suppose that $D\subseteq T$ is a totally ordered subset of limit order type of an order tree $(T,\leq)$. Then, there is an unique $\subseteq-$minimal ray $R_D$ in $T$ containing $D$. Moreover, $D$ is cofinal in $R_D$. 
\end{lemma}
\begin{proof}
	Consider $R_D:=\bigcup_{t\in D}\lceil t\rceil$, so that this is a down-closed subset of $T$ by being the union of other down-closed sets. Similarly, $D$ is cofinal in $R_D$ by its own definition. In addition, any ray $R$ in $T$ containing $D$ must also contain $R_D$, once $R$ is down-closed as well. For any pair $s,t\in D$, the fact that $D$ is totally ordered now implies that $\lceil s \rceil \subseteq \lceil t\rceil $ or $\lceil t\rceil \subseteq \lceil s \rceil$, depending on whether $s<t$ or $t<s$. In other words, $R_D$ is also a totally ordered subset of $T$ and, thus, defines now a path in $T$. Showing that it is actually a ray, $R_D$ has limit order-type regarding $\leq$ because so does $D$ (which is cofinal in $R_D$). 
\end{proof}

As a central idea within this section, we shall now formalize how some order-isomorphic copies of $\omega$ can be pruned from an order tree $T$ without loss of major structural and topological properties. To that aim, we may say that $t\in T\setminus \{t_0\}$ is a \textbf{tail node} of a ray $R \in \mathcal{R}(T)$ if $X_R^t:=R\setminus \lceil t\rceil$ has order type $\omega$ and
 $X_R^t\cup \{t\}$ contains no branching node of $T$. Below, we highlight some first consequences of this definition:
 
 \begin{lemma}\label{TailNode}
 	Let $(T,\leq)$ be an order tree and $t\in T$ be a tail node of a ray $R\in \mathcal{R}(T)$. Then,
 	\begin{itemize}
 		\item[$a)$] every other element from $X_R^t= R\setminus \lceil t\rceil $ is also a tail node of $R$;
 		\item[$b)$] $R\subseteq R'$ for every ray $R'\in \mathcal{R}(T)$ containing $t$;
 		\item[$c)$] no node from $X_R^t$ has limit height in $T$.  
 	\end{itemize}
 \end{lemma}
\begin{proof}
	
	By definition of tail node, we may fix an enumeration of order-type $\omega$ for $X_R^t$ as $t_1<t_2<\dots$, so that $t_{n+1}$ is a successor of $t_n$ for each $n\geq 1$ because $R$ is down-closed. In addition, $R\setminus \lceil t_n\rceil = \{t_{n+1}<t_{n+2}<\dots\}$ has also order-type $\omega$ and, being a subset of $X_R^t$, contains no branching nodes of $T$. In other words, proving item $a)$, each $t_n$ is also a tail node of $R$. On the other hand, note that $t_1$ is a successor of $t$ also because $R$ is down-closed and contains this later node. In particular, $t_1\in R'$ for every other ray $R'\in\mathcal{R}(T)$ containing $t$, which, as a tail node, has no successor in $T$ other than $t_1$. Actually, if $t_n \in R'$ for some $n\geq 1$, then $t_{n+1}\in R'$ because $R'$ has no $\leq-$maximum element and $t_{n+1}$ is the unique successor in $T$ of $t_n$ (which is not a branching node of $T$ by belonging to $X_R^t$). Concluding item $b)$, this proves the inclusion $R\subseteq R'$ by induction on $n\in\mathbb{N}$.

	Finally, suppose for a contradiction that some $s\in X_R^t$ has limit height in $T$, thus defining the ray $R':=\mathring{\lceil s\rceil}$. Hence, $t\in R'$ because, equivalently, $s \in X_R^t = R\setminus \lceil t\rceil$ is greater than $t$ regarding $\leq$. In this case, the previous paragraph just asserted that $R\subseteq R'$, contradicting the own definition of $R'$ as a ray not containing $s$. Therefore, item $c)$ must also hold.
	
\end{proof}

As a particular consequence of item $b)$ from the above result, note that some $t\in T$ is a tail node of \textit{at most} one ray $R_t$ of the given order tree $T$. Indeed, $R_t$ is characterized in this case as the $\subseteq-$minimal element from $\mathcal{R}(T)$ containing $t$. Considering this observation (and notation), we now formalize the following construction:

\begin{lemma}\label{Podando}
	Let $(T,\leq)$ be an order tree rooted at some $t_0\in T$. Fix also a subset $K\subseteq T\setminus \{t_0\}$ comprising tail nodes in $T$ such that $R_t\neq R_s$ if $s,t\in K$ are distinct. Now, consider the tree $\hat{T}$ obtained from $T$ after restricting its ordering $\leq$ to the set of nodes $T\setminus \bigcup_{t\in K}X_{R_t}^t$. Then, the following properties are verified:
	\begin{itemize}
		\item[$a)$] Nodes of $T$ which are branching or have limit height belong to $\hat{T}$;
		\item[$b)$] The map $\psi: \mathcal{R}(T)\to \mathcal{P}(\hat{T})$ defined by $\psi(R) = R\cap \hat{T}$ for every $R\in \mathcal{R}(T)$ is an homeomorphism onto its image, which comprises the disjoint union $\mathcal{R}(\hat{T})\cup \{\lceil t\rceil_{\hat{T}} : t\in K\}$;
		\item[$c)$] Each ray $\hat{R}\in \mathcal{R}(\hat{T})$ is cofinal in $\psi^{-1}(\hat{R})$ and contains infinitely many branching nodes of $T$ or infinitely many nodes of limit height in this later tree.
	\end{itemize}
\end{lemma} 
\begin{proof}
    First, note that item $a)$ holds because, by definition of tail nodes, $X_{R_t}^t$ does not contain branching nodes or nodes of limit height in $T$ for every $t\in K$. In addition, once the ordering of $\hat{T}$ corresponds to a restriction of the ordering of $T$ to a set of nodes containing $t_0$, it clearly follows that $\hat{T}$ is also a tree rooted at $t_0$. More than that, since rays are down-closed and totally ordered subsets of $T$, the intersection $T\cap R$ indeed defines a path of $\hat{T}$ for every ray $R\in \mathcal{R}(T)$. In other words, $\psi : \mathcal{R}(T)\to \mathcal{P}(\hat{T})$ as introduced by item $b)$ is a well-defined map. For each ray $R\in \mathcal{R}(T)$, we now describe $\psi(R)$ according to one of the following cases:
    
    \begin{itemize}
    	\item If $R\in \mathcal{R}(T)$ cannot be written as $R_t$ for any $t\in K$, then the first two bullet points in the definition of tail nodes claim that for each $t \in R\cap K$ there is $s \in R\setminus \lceil t\rceil_T$ which is either a top of a ray (thus having limit height in $T$) or a branching node of $T$. In both case, we must have $s\in \hat{T}$ by item $a)$. Therefore, $\psi(R) = R\cap \hat{T}$ has limit order type and, thus, $\psi(R)\in \mathcal{R}(\hat{T})$. Conversely, if $\hat{R}\in \mathcal{R}(\hat{T})$ is a ray of $\hat{T}$, it is clearly also a totally ordered subset of $T$ having limit order type. Then, Lemma \ref{InclusionwiseMinimalRay} claims that there is an unique ray $R\in \mathcal{R}(T)$ which is inclusionwise minimal for the property of containing $\hat{R}$. The same result also shows that $\hat{R}$ is cofinal in $R$, from where we conclude the equalities $\hat{R}= R\cap \hat{T} = \psi(R)$. In particular, if $R,S\in \mathcal{R}(T)$ are rays such that $\psi(R) = \psi(S)$, the inclusionwise minimality of both $R$ and $S$ for containing $R\cap \hat{T} = S\cap \hat{T}$ implies that $R=S$;   
    	\item If $R\in \mathcal{R}(T)$ can be written as $R=R_t$ for some $t\in K$, then there is no $s\in R\cap \hat{T}$ for which $s> t_R$. After all, $X_{R_t}^t=R\setminus \lceil t\rceil_T$ is disjoint from $\hat{T}$ by the own definition of this later tree. Therefore, $t$ is the $\leq-$maximum element from $R\cap \hat{T}$ or, equivalently, $\psi(R) = R\cap \hat{T} = \lceil t\rceil_{\hat{T}}$. In particular, if $R_t$ and $R_s$ are rays admitting distinct tail nodes $s,t\in K$, then $\psi(R_t)= \lceil t\rceil_{\hat{T}} \neq \lceil s\rceil_{\hat{T}} = \psi(R_s)$. 
    \end{itemize}

To summarize, the above two items claim that $\psi: \mathcal{R}(T)\to \mathcal{P}(\hat{T})$ is an injective map whose range is given by $\mathcal{K}:=\mathcal{R}(\hat{T})\cup \{\lceil t\rceil\cap \hat{T}: t\in K\}$. On the other hand, in order to indeed conclude item $b)$, recall that a sub-basic open set in $\mathcal{P}(\hat{T})$ may be written as $[t]_{\hat{T}}$ or $\mathcal{P}(\hat{T})\setminus [t]_{\hat{T}}$ for some $t\in \hat{T}$. Naturally, for each ray $R\in \mathcal{R}(T)$, we then have $\psi(R)\in [t]_{\hat{T}}$ if, and only if, $t\in R$. Since $[t,\emptyset]_T:=\{R\in \mathcal{R}(T): t\in R\}$ and $\mathcal{R}(T)\setminus [t,\emptyset]_T$ are both open sets for the ray space topology of $\mathcal{R}(T)$, as highlighted before Theorem \ref{RepresentacaoPitz}, this previous equivalence shows that $\psi$ is continuous.

Conversely, a basic open set for the topology of $\mathcal{R}(T)$ around a ray $R$ is again written as $[t,F]_T$ for some $t\in R$ and some set $F\subseteq T$ comprising finitely many tops of $R$. In particular, $F\subseteq \hat{T}$ by property item $a)$, once tops of rays have limit height in $T$. If $t\in X_{R_{t'}}^{t'}$ for some $t'\in K$, we recall from item $b)$ in Lemma \ref{TailNode}  that a ray $R'\in \mathcal{R}(T)$ contains $t$ if, and only if, it also contains $R_{t'}$ (and thus $t'$). In other words, since $t'\in \hat{T}$, we can describe the open set $[t,F]_T$ while assuming that $t$ belongs to $\hat{T}=T\setminus \bigcup_{t'\in K}X_{R_{t'}}^{t'}$. In this case,

\begin{align*}
	\psi([t,F]_T) & = \{\psi(P): P\in \mathcal{R}(T)\text{ such that } t\in P\text{ and } s\notin P\text{ for every }s\in F\} \\ & = \{P'\in \mathcal{K}: t\in P'\text{ but }s\notin P'\text{ for every }s\in F\}\\ & = \left([t]_{\hat{T}}\setminus \bigcup_{s\in F}[s]_{\hat{T}}\right)\cap \mathcal{K}.
\end{align*}

 Once $[t]_{\hat{T}}\cup \{[s]_{\hat{T}}: s \in F\}$ is a collection of finitely many clopen sets for the topology of $\mathcal{P}(T_K)$, it follows that $\psi([t,F]_T)$ is then open in the subspace topology of $\mathcal{K}$. This concludes that $\psi: \mathcal{R}(T)\to \mathcal{K}$ is an homeomorphism.

Finally, in order verify item $c)$, fix a ray $\hat{R}\in \mathcal{R}(\hat{T})$ and denote $R = \psi^{-1}(\hat{R})\in \mathcal{R}(T)$. Suppose for instance that there is $t\in \hat{R}$ such that $\hat{R}\setminus \lceil t\rceil$ contains no branching nodes of $T$ nor vertices of limit height in this later tree. Since $\hat{R} = R\cap T$, it follows from item $a)$ that $R\setminus \lceil t\rceil_T$ itself contains no branching nodes nor nodes of limit height in $T$. In other words, $t$ is a tail node of $R$ and, thus, we may write $R = R_s$ for some $s\in K$. In this case, $\psi(R) = \lceil s\rceil_{\hat{T}}$ by the previous second bullet point, contradicting the choice of $\hat{R} = \psi(R)$ as a ray in $\hat{T}$. Hence, the given node $t\in \hat{R}$ must not exist, which particularly concludes that $\hat{R}$ contains infinitely many branching nodes of $T$ or infinitely many nodes of limit height within this later tree. Then, since $\hat{R}$ is cofinal in $R$ by the previous first bullet point, the proof of item $c)$ is also complete.

\end{proof}

As motivated at the beginning of this section, we can now briefly revisit a construction due to Pitz in \cite{caracterizacao} which supports the proof his Theorem 3.9. To that aim, let $(T,\leq)$ be a given order tree and fix a subspace $\mathcal{K}\subseteq \mathcal{P}(T)$ containing $\mathcal{R}(T)$. We shall construct a second order tree $T_{\mathcal{K}}$ as an extension of $(T,\leq)$ such that $\mathcal{R}(T_{\mathcal{K}})\simeq \mathcal{K}$. If $P$ is a path from $ \mathcal{K}':=\mathcal{K}\setminus \mathcal{R}(T)$, denote first $t_P:=\max P$ and fix an order-isomorphic copy of $\omega$ (disjoint from $T$) written as $\mathbb{N}_P:=\{t_P^0<t_P^1<t_P^2<\dots\}$. Then, $T\cup \bigcup_{P\in\mathcal{K}'}\mathbb{N}_P$ defines $T_{\mathcal{K}}$ as a set. In this case, the ordering $\leq$ of $T$ is extended to $T_{\mathcal{K}}$ according to the following items:

\begin{itemize}
	\item Given $P\in \mathcal{K}'$, we set $s< t_P^n$ for each $s\in P$ and each $n\in\mathbb{N}$. In particular, $t_P^0$ is a successor of $t_P$ and $t_{P}^{n+1}$ is a successor of $t_P^n$ for every $n\in\mathbb{N}$;
	\item Given $P\in \mathcal{K}'$ and $t\in T$ such that $t>t_P$, we set $t>t_P^n$ for each $n\in\mathbb{N}$.  
\end{itemize}

As an immediate consequence of the above definition, the root $t_0$ of $T$ is also the root of $T_{\mathcal{K}}$. In addition, for each $P\in \mathcal{K}'$ and $n\geq 1$, the nodes $t_{P}^{0}$ and $t_P^n$ are the unique successors of $t_{P}$ and $t_P^{n-1}$ in $T_{\mathcal{K}}$ respectively. In particular, this shows that branching nodes of $T_{\mathcal{K}}$ belong to $T$, as well as those nodes of $T_{\mathcal{K}}$ whose height is a limit ordinal. In other words, $t_P$ is a tail node of the unique $\subseteq-$minimal ray $R_P$ in $T_{\mathcal{K}}$ containing $\mathbb{N}_P$, as formalized by Lemma \ref{InclusionwiseMinimalRay}. After all, the ``moreover-''part of this result ensures that $\mathbb{N}_P=R_P\setminus \lceil t_P\rceil_{T_{\mathcal{K}}}$. Hence, since $T$ may be recovered from  $T_{\mathcal{K}}$ after restricting its order to the set $T_{\mathcal{K}}\setminus \bigcup_{P\in \mathcal{K}'}\mathbb{N}_P$, we finish this section by highlighting the following application of Lemma \ref{Podando}:

\begin{corol}
	Let $(T,\leq)$ be an order tree and fix a subspace $\mathcal{K}\subseteq \mathcal{P}(T)$ containing the rays of $T$. Then, $\mathcal{R}(T_{\mathcal{K}})$ is homeomorphic to $\mathcal{K}$. 
\end{corol}
\begin{proof}[Revisited proof of Theorem 3.9 in \cite{caracterizacao}]
	For each $P\in \mathcal{K}':=\mathcal{K}\setminus \mathcal{R}(T)$, fix the node $t_P:=\max P$ and the order-isomorphic copy $\mathbb{N}_P$ of $\omega$ as in the construction of $T_{\mathcal{K}}$. Recall from the previous paragraph that $t_P$ is a tail node of the unique $\subseteq-$minimal ray $R_P$ in $T_{\mathcal{K}}$ containing it and, then, $\mathbb{N}_P=  R_{P}\setminus \lceil t_P\rceil_{T_{\mathcal{K}}}$. Therefore, $T$ may be given by $ T_{\mathcal{K}}\setminus \bigcup_{P\in \mathcal{K}'}(R_P\setminus \lceil t_P\rceil_{T_{\mathcal{K}}})$. On the other hand, once $P = \lceil t_P\rceil_T$ for every $P\in \mathcal{K}'$ and $\mathcal{R}(T)\subseteq \mathcal{K}$, we can write $\mathcal{K} = \mathcal{R}(T)\cup \mathcal{K}'=\mathcal{R}(T)\cup \{\lceil t_P\rceil_T: P \in \mathcal{K}'\}$. Due to this last expression, the statement is now verified by the homeomorphism $\psi : \mathcal{R}(T_{\mathcal{K}})\to \mathcal{K}\subseteq \mathcal{P}(T)$ as defined within item $b)$ of Lemma \ref{Podando}. 
\end{proof}

\section{Proof of the main result}\label{FinalProof}

\paragraph{}
This final section discusses how the results summarized so far throughout this paper support a short proof for Theorem \ref{main}, which states a topological description for the family $\Omega_E=\{\Omega_E(G): G \text{ graph}\}$. We shall split our approach into two intermediate characterizations, closely following the inspiring literature on end spaces. As a first step, the correspondence below may be read as a restriction of Theorem \ref{RepresentacaoPitz} to the study of edge-ends:

\begin{prop}\label{MinhaRepresentacao}
	A topological space $X$ is homeomorphic to the edge-end space of some graph if, and only if, it is homeomorphic to the ray space of an order tree $T$ whose height is bounded by $\omega\cdot \omega$ and whose branching rays contain only finitely many branching nodes.
\end{prop}

\begin{proof}
	If $T$ is a graph-theoretical tree and $v_0\in V(T)$ is a fixed root, let $\leq$ denote the corresponding tree-order as recalled in Section \ref{sec:GraphTrees}. In this case, Lemma \ref{PathTopologies} describes an homeomorphism $\sigma : \mathcal{P}(T)\to \|T\|$ that particularly identifies $\mathcal{R}(T)$ and $\Omega_E(T)$. Hence, any subspace $K\subseteq \|T\|$ containing $\Omega_E(T)$ is homeomorphic to a subspace $\mathcal{K}\subseteq \mathcal{P}(T)$ containing $\mathcal{R}(T)$. Now, Section \ref{sec:order} provides a construction of an order tree $T_{\mathcal{K}}$ which extends the ordering of $T$ and whose ray space is homeomorphic to $\mathcal{K}$. Revisiting the definition of $T_{\mathcal{K}}$ more carefully, the properties below verify a first implication from the statement:
	
	\begin{itemize}
		\item If $v$ has limit height in $T_{\mathcal{K}}$, then it actually belongs to $T$. Nevertheless, $\lceil v \rceil_T$ contains finitely many vertices, since $T$ is a graph-theoretical tree. This implies that $\lceil v \rceil_{T_{\mathcal{K}}}$ contains only finitely many vertices of limit height in $T_{\mathcal{K}}$. If $u$ is such a vertex, the order-type of $\mathring{\lceil u\rceil}_{T_{\mathcal{K}}}$ (which is down-closed in $T_{\mathcal{K}}$) reads as $\omega\cdot n$ for some $n\in\mathbb{N}$. In other words, the height of $v$ in $T$ is strictly bounded by $\omega\cdot \omega$, from where we obtain that this ordinal is also an upper bound for the own height of $T_{\mathcal{K}}$;
		\item Similarly, if $R\in \mathcal{R}(T_{\mathcal{K}})$ has a top $v$, then $v\in V(T)$ because $v$ has limit height in $T_{\mathcal{K}}$. In addition, branching nodes of $T$ also belong to $T_{\mathcal{K}}$ by construction, so that $R$ must contain at most finitely many of them. Otherwise, they would define an infinite subset of $\lceil v\rceil_T$, contradicting again that the ordering $(T,\leq)$ arises from a graph-theoretical tree (in which, hence, $v$ has finite height).  
	\end{itemize}

	Conversely, let $T$ be now an order tree whose height is bounded by $\omega\cdot \omega$ and whose branching rays contain only finitely many branching nodes. Then, denote by $\hat{T}$ the order tree obtained from $T$ as described by Lemma \ref{Podando}, so that $\mathcal{R}(T)$ is homeomorphic to a subspace of $\mathcal{P}(\hat{T})$ containing $\mathcal{R}(\hat{T})$. Combining Theorem \ref{RepresentacaoEdgeEndsPitz} with Lemma \ref{PathTopologies}, the proof is finished after showing how $(\hat{T},\leq)$ arises as the ordering of a graph-theoretical tree. According to Lemma \ref{BasicPropertiesTree}, this is equivalent to verifying that the height of $\hat{T}$ is bounded by $\omega$.

	In fact, suppose for a contradiction that there is a node $t\in \hat{T}$ whose height in $\hat{T}$ is a limit ordinal. In this case, $\hat{R}:=\mathring{\lceil t\rceil}_{\hat{T}} $ defines a ray of $\hat{T}$ which admits $t$ as a top. Due to this last property, note that $\hat{R}$ cannot contain infinitely many nodes of limit height in $T$: after all, the height of $t$ in $T$ should be at least $\omega\cdot \omega$ otherwise. Since this contradicts the assumption of $\omega\cdot \omega$ as an upper bound for the own height of $T$, item $c)$ from Lemma \ref{Podando} claims that $\hat{R}$ contains infinitely many branching nodes from $T$. Nevertheless, item $b)$ in the same result also provides a ray $R\in \mathcal{R}(T)$ whose intersection with $\hat{T}$ is precisely $\hat{R}$. Since $\hat{R}$ is even cofinal in $R$, it follows that $t > s$ for every $s\in R$ and, thus, that $R$ is a branching ray of $T$. In fact, $t$ is actually a top of $R$ in $T$ by item $a)$ within Lemma \ref{Podando} and since $R\cap \tilde{T} = \hat{R} = \mathring{\lceil t\rceil}_{\hat{T}}$. However, $R$ now contradicts the assumption that $T$ admits no branching ray containing infinitely many branching nodes. Therefore, no node has limit height in $(\hat{T},\leq)$ and, equivalently, the own height of $\hat{T}$ is bounded by $\omega$.          
\end{proof}

In its turn, the next result sets an instance of Theorem \ref{caracterizacaoPitz} that characterizes the family of ray spaces just highlighted by Proposition \ref{MinhaRepresentacao}. Therefore, together with this later statement, the correspondence below finishes the proof of our Theorem \ref{main} by asserting the claimed topological description for the family of edge-end spaces:  

\begin{prop}\label{CaracterizacaoDeFato}
	The following two items are equivalent for a given topological space $X$:
	\begin{itemize}
		\item[$i)$] $X$ is homeomorphic to the ray space of an order tree whose height is bounded by $\omega\cdot \omega$ and whose branching rays contain finitely many branching nodes;
		\item[$ii)$] $X$ admits a clopen subbase which is nested, noetherian, hereditarily complete and verifies the singleton intersection property. 
	\end{itemize}
\end{prop}     
\begin{proof}
	Suppose first that $X$ is homeomorphic to the ray space of an order tree $T = (T,\leq)$. Identifying $X$ with $\mathcal{R}(T)$, we recall from Section \ref{sec:background} that a clopen subbasic set in $X= \mathcal{R}(T)$ may be written as $[t]:=\{P\in \mathcal{R}(T): t\in P\}$ for some $t\in T$. In this case, the family $\mathcal{C}:=\{[t]: t\in T\}$ is noetherian and hereditarily complete by Propositions 2.10 and 2.13 in \cite{caracterizacao} respectively. Moreover, it is easily seen that this clopen subbase for $X$ is nested. In fact, since rays are down-closed and totally ordered subsets of $T$, we have $[t]\cap [s] = \emptyset$ if $t,s\in T$ are incomparable regarding $\leq$, while $[t]\subseteq [s]$ if $s\leq t$.

	In order to show that $\mathcal{C}$ verifies the singleton intersection property, let $\{t_n\}_{n\in\mathbb{N}}$ be a sequence of nodes of $T$ such that $[t_n]\supsetneq [t_{n+1}]$ for each $n\in\mathbb{N}$. In other words, $t_n > t_{n+1}$ for every such $n$ by the nestedness of $\mathcal{C}$, from where Lemma \ref{InclusionwiseMinimalRay} ensures the existence of a $\subseteq-$minimal ray $R\in \mathcal{R}(T)$ containing $\{t_n\}_{n\in\mathbb{N}}$. Then, $t_n\in R$ for every $n\in\mathbb{N}$ or, equivalently, $R\in \bigcap_{n\in\mathbb{N}}[t_n]$. For a contradiction, fix some $R'\in \bigcap_{n\in\mathbb{N}}[t_n]$ distinct from $R$, so that $R\subsetneq R'$ by the $\subseteq-$minimality of $R$. In this case, the node $\min R'\setminus R$ verifies that $R$ is a branching ray of $T$, which then contains only finitely many branching nodes by hypothesis. Hence, we may choose a large enough $n_0\in\mathbb{N}$ for which $R\setminus \lceil t_{n_0}\rceil$ contains no branching node. Similarly, once the order type of $R$ is not greater than $\omega\cdot \omega$ (once this ordinal is an upper bound for the own height of $T$ by assumption), only finitely many nodes of $R$ have limit height in $T$. In other words, $n_0$ can be chosen to be big enough so that $R\setminus \lceil t_{n_0}\rceil$ has order-type $\omega$. Hence, $t_n$ is a tail node of $R$ for each $n\geq n_0$ and, thus, $[t_n] = [t_{n_0}]$. Since this contradicts the choice of $\{[t_n]\}_{n\geq n_0}$ as a $\subseteq-$strictly decreasing sequence in $\mathcal{C}$, we must have $\bigcap_{n\in\mathbb{N}}[t_n] = \{R\}$. Therefore, finishing the proof of $i)\Rightarrow ii)$, the singleton intersection property holds for $\mathcal{C}$.

	Conversely, suppose that $X$ admits a clopen subbase $\mathcal{C}$ which is nested, noetherian and hereditarily complete as above. Then, by Theorem  \ref{caracterizacaoPitz}, there is a $\mathcal{C}-$tree $T = (T,\leq)$ whose associated map $f: T\to \mathcal{C}$ can be used to define an homeomorphism $e: X\to \mathcal{R}(T)$ after setting $e(x):=\{t\in T: x\in f(t)\}$ for every $x\in X$. Now assuming the singleton intersection property for $\mathcal{C}$, we will argue that the height of $T$ is bounded by $\omega\cdot \omega$ and that its branching nodes contains only finitely many branching vertices.

	For a contradiction, suppose that some node $t\in T$ has height $\omega \cdot \omega$ and fix the ray $R\in \mathcal{R}(T)$ given by $R = \mathring{\lceil t\rceil}$. In particular, for each $n\in\mathbb{N}$ there is a node $t_n \in R:=\mathring{\lceil t\rceil}$ whose height is $\omega\cdot n$, from where $t_0$ is then chosen as the root of $T$. On the other hand, each $t_n$ with $n\geq 1$ is a top of the ray $R_n:=\mathring{\lceil t_n\rceil}$, whose inverse image under $e$ then belongs to $f(t_{n-1})$ (since $t_{n-1}\in R_n$) but not to $f(t_n)$ (because $t_n\notin R_n$). Hence, $f(t_n)\supsetneq f(t_{n+1})$ for every $n\in\mathbb{N}$, so that $\bigcap_{n\in\mathbb{N}}f(t_n)$ comprises a single element $x\in X$ the singleton intersection property of $\mathcal{C}$. On one hand, we should have $e(x) = R$ due to the fact that $R = \mathring{\lceil t\rceil} $ is a ray of $T$ and $t_n \in R$ for every $n\in\mathbb{N}$. On the other hand, $f(t)\subseteq \bigcap_{n\in\mathbb{N}}f(t_n)$ by definition of $\mathcal{C}-$tree and because $t>t_n$ for every $n\in\mathbb{N}$. The uniqueness of the intersection $\bigcap_{n\in\mathbb{N}}f(t_n)$ then implies that $x\in f(t)$, since $f(t)$ is a non-empty set from $\mathcal{C}$. However, we should have $t\in e(x) = R$ by the law that sets $e$, contradicting the choice of $t$ as a top of $R$. Therefore, $T$ has no nodes of height $\omega\cdot \omega$ and, hence, this ordinal is an upper bound for the own height of $T$.

	Analogously, suppose that $t\in T$ is the top of a ray $R=\mathring{\lceil t\rceil}$ and that $\{t_n\}_{n\in\mathbb{N}}\subseteq R$ is a sequence comprising infinitely many branching nodes. Since $R$ is well-order by $\leq$, we may clearly assume that $t_n < t_{n+1}$ for every $n\in\mathbb{N}$, so that $f(t_{n})\supseteq f(t_{n+1})$ again by the definition of $\mathcal{C}-$tree. However, once $t_n$ is a branching node of $T$, there is a successor $s_n'$ of $t_n$ other than the one lying on $R$, which we may denote by $s_n$. Once $t_n<s_n\leq t_{n+1}$ and $t_n<s_n'$, it is also immediate from the definition of $\mathcal{C}-$tree that $f(t_n)\supseteq f(s_n)\supseteq f(t_{n+1})$ and $f(t_n)\supseteq f(s_n')$. On the other hand, $f(s_n)\cap f(s_n') = \emptyset$ because $s_n$ and $s_n'$ are incomparable regarding $\leq$ by being two different successors of $t_n$. Hence, $f(t_n)\supsetneq f(t_{n+1})$ since $f(s_n')$ is a non-empty clopen set from $\mathcal{C}$. Therefore, $\bigcap_{n\in\mathbb{N}}f(t_n)$ contains once more an unique element $x\in X$ by the singleton intersection property of $\mathcal{C}$. Again because $t$ is a top of $R$ and $\{t_n\}_{n\in\mathbb{N}}\subseteq R$, we should have $e(x) = R$  as well as $f(t)\subseteq \bigcap_{n\in\mathbb{N}}f(t_n)$, then concluding that $f(t) = \{x\}$. By definition of $e$, this contradicts the choice of $t$ as a node not belonging to $e(x) = R = \mathring{\lceil t\rceil}$. Therefore, finishing the proof of $ii)\Rightarrow i)$, the branching ray $R$ must contain only finitely many branching nodes of $T$.

\end{proof}

\section{Acknowledgments}
\paragraph{}

I acknowledge the São Paulo Research Foundation (FAPESP) for the financial support through grant number 2024/01158-1.

\bibliography{NoteEdgeEndSpaces}
\bibliographystyle{plain}

\end{document}